\documentclass[11pt]{amsart}
\usepackage[T1]{fontenc}

\newcommand{\mfg}{\mathfrak{g}}
\newcommand{\mfu}{\mathfrak{u}}

\newcommand{\mfe}{\mathfrak{e}}

\newcommand{\mfS}{\mathfrak{S}}
\newcommand{\mbE}{\mathbb{E}}

\newcommand{\NN}{\mathbb{N}}

\newcommand{\msW}{\mathscr{W}}

\newcommand{\mrlie}{\mathrm{Lie}}
\newcommand{\mrmax}{\mathrm{Max}}
\newcommand{\mrcom}{\mathrm{Com}}
\newcommand{\mrlt}{\mathrm{LT}}
\newcommand{\mrstab}{\mathrm{Stab}}
\newcommand{\mrrad}{\mathrm{rad}}
\newcommand{\mrsp}{\mathrm{Span}}
\newcommand{\mmax}{\mathrm{max}}

\newcommand{\Phir}{\Phi^{\mrrad}}
\newcommand{\Bk}{\Bbbk}
\newcommand{\msrke}{\mathsf{rk}_p}
\newcommand{\set}[1]{\left\{#1\right\}}

\def\exp{\operatorname{exp}\nolimits}
\def\ad{\operatorname{ad}\nolimits}

\def\odd{\operatorname{odd}\nolimits}
\def\ev{\operatorname{ev}\nolimits}

\def\Gr{\operatorname{Gr}\nolimits}

\usepackage{pictex}
\usepackage{amscd}
\usepackage{latexsym}
\usepackage{amssymb}
\usepackage{eucal}
\usepackage{eufrak}
\usepackage{verbatim}
\usepackage[all]{xy}
\usepackage{multirow}
\usepackage{mathrsfs}

\usepackage{hyperref}
\usepackage[all]{hypcap}
\hypersetup{backref,colorlinks=true}

\numberwithin{equation}{section}

\newtheorem{Theorem}{Theorem}[section]

\theoremstyle{Theorem}
\newtheorem{Thm}{Theorem}[subsection]
\newtheorem{Lem}[Thm]{Lemma}

\newtheorem{Cor}[Thm]{Corollary}

\newtheorem*{thm*}{Theorem}
\newtheorem*{thm**}{Corollary}
\newtheorem*{thm***}{Theorem B} 

\theoremstyle{remark}
\newtheorem*{Remark}{Remark}
\newtheorem*{Remarks}{Remarks}
\newtheorem*{Definition}{Definition}

\newtheorem*{Notation}{Notation}
\numberwithin{equation}{section}

\begin{document}

\title{Varieties of elementary subalgebras of submaximal rank in type A}
\author[Yang Pan]{Yang Pan}
\address{School of Sciences, Zhejiang A\&F University, 311300 Hangzhou, Zhejiang, China}
\address{Mathematisches Seminar, Christian-Albrechts-Universit\"at zu Kiel, Ludewig-Meyn-Str.4, 24098 Kiel, Germany}
\email{ypan@outlook.de}
\date{\today}

\begin{abstract} 
Let $G$ be a connected simple algebraic group over an algebraically closed field 
$\Bk$ of characteristic $p>0$, and $\mfg:=\mrlie(G)$. We additionally assume that 
$G$ is standard and is of type $A_{n}$.  Motivated by the investigation of the geometric properties of the
varieties $\mbE(r,\mfg)$ of $r$-dimensional elementary subalgebras of a restricted Lie algebra $\mfg$, 
we will show in this article the irreducible components of $\mbE(\msrke(\mfg)-1,\mfg)$ when
$\msrke(\mfg)$  is the maximal dimension of an elementary subalgebra of $\mfg$.
\end{abstract}

\maketitle

\section*{Introduction}
Let $(\mfg,[p])$ be a finite dimensional restricted Lie algebra over an algebraically closed field $\Bk$ 
of positive characteristic $p>0$.  The closed subset of $p$-restricted nilpotent elements 
\[  V(\mfg) := \{ x\in \mfg \ ; \ x^{[p]}= 0 \}  \]
has been studied in the modular representation theory of $(\mfg, [p])$, 
aiming to understand the cohomological support variety.
A Lie subalgebra $\mfe \subset \mfg$ is said to be elementary if it is abelian and has trivial $p$-restriction.   
The subset 
\[  \mbE(r,\mfg)  : = \{   \mfe \in \Gr_r(\mfg)   \ ; \   [\mfe, \mfe]=0,  \, \mfe \in V(\mfg)  \}       \]
of the Grassmannian $\Gr_r(\mfg)$ of $r$-planes in $\mfg$ which consists of $r$-dimensional elementary subalgebras of $\mfg$
has been expounded in \cite{CFP} by Carlson, Friedlander and Pevtsova.
For instance, they show us $\mbE(r,\mfg)$ can be endowed with  a projective variety structure and it
affords geometric invariant for the representations of $\mfg$.

When concerning the geometric properties of $\mbE(r, \mfg)$,  interest has been shown in determining
its irreducible components.  A prototypical example arises from $\mbE(1,\mfg)$, which is the projectivization 
of the restricted nullcone $V(\mfg)$.  When $\mfg$ is the Lie algebra of a simple algebraic group, 
$V(\mfg)$ is irreducible regardless of the characteristic $p$, so is $\mbE(1, \mfg)$. 
When $r$ equals 2, Premet in \cite{Pre1} shows the correspondence between the irreducible components of the nilpotent
commuting variety $\mathcal{C}^{nil}(\mfg)$ and the distinguished nilpotent orbits of $\mfg$ when $\mfg$ is a 
reductive Lie algebra. It follows that $\mathcal{C}^{nil}(\mfg)$ is irreducible when $\mfg$ is of 
type $A_{n}$,  in which case the same is true of $\mbE(2,\mfg)$ if $p\geq n+1$.  Let
\[ \msrke(\mfg) := \max \{ r \in \NN_0 \ ; \ \mbE(r,\mfg) \ne \emptyset\} \]
be the $p$-rank of $\mfg$. This rank of the restricted Lie algebra of a simple algebraic group was determined 
earlier in \cite{CFP} and in recent work by Pevtsova-Stark in \cite{PS}.  Irreducible components
of $\mbE(\msrke(\mfg), \mfg)$ for these Lie algebras were calculated case by case and were shown in  
\cite[Table 4]{PS}.

It is the purpose of this article to give a description of the variety $\mbE(\msrke(\mfg)-1,\mfg)$ when
$\mfg$ is the Lie algebra of a connected standard simple algebraic group $G$ of type $A$. 
Under the standard assumption one can show that $\mfg$ is a Lie algebra isomorphic to $\mathfrak{sl}_{n}(\Bk)$
such that $p$ does not divide $n$. In view of \cite[Lemma 2.2]{Pre2}, 
the determination of $\mbE(\msrke(\mfg)-1,\mfg)$ can be reduced to the unipotent case 
$\mbE(\msrke(\mfg)-1,\mfu)$ 
where $\mfu=\mrlie(R_{u}(B))$ and $R_{u}(B)\subset B \subset G$ is 
the unipotent radical of a fixed Borel subgroup $B$ of $G$. 
Let $\Phi$ be an irreducible root system with positive roots $\Phi^+$.
Recall from \cite{PS} that two positive roots commute if their sum is not a root.
We define the set
\[ 
   \begin{gathered} \mrmax_{r}(\Phi):=
  \end{gathered}  
  \left\{
  \begin{gathered} R \subset \Phi^+  \ ; \ \alpha+\beta \notin \Phi^+ , \forall  \alpha,\beta\in R, |R| =r \;\mbox{and}\;
   R\not\subset R^{'} \\ 
  \mbox{where}\; R^{'} \mbox{is any subset of commuting positive roots}
\end{gathered}
\right\}.
\]
When $r=\msrke(\mfg)$, we write $\mrmax_{r}(\Phi)$ as $\mrmax(\Phi)$ simply.
By considering the set 
\[  \mrcom_{r}(\Phi):= \{ R \subset \Phi^+  \ ; \ \alpha+\beta \notin \Phi^+ , \forall  \alpha,\beta\in R, |R| =r \}.\]
we find the map $ \xymatrix{ \mrlt:\mbE(\msrke(\mfg),\mfu) \ar@{->}[r] & \mrmax(\Phi) }$ in \cite[(3.1.2)]{PS}
can be defined in a generalized fashion
$ \xymatrix{ \mrlt:\mbE(r,\mfu) \ar@{->}[r] & \mrcom_{r}(\Phi) }$ since 
$\mrcom_{r}(\Phi)=\mrmax(\Phi)$ when $r=\msrke(\mfg)$.
The problem now is for any given total ordering and any element $\mfe$ of $\mbE(\msrke(\mfg)-1,\mfu)$,
it is possible to have $\mrlt(\mfe)\notin \mrmax_{\msrke(\mfg)-1}(\Phi)$.
Let $\mbE(\msrke(\mfg)-1,\mfu)_{\mmax}$ be the subset of $\mbE(\msrke(\mfg)-1,\mfu)$ consisting of 
maximal elementary subalgebras.
This raises the question concerning the ordering on $\Phi^+$,  the one giving rise to the map
\[ \xymatrix{\mrlt: \mbE(\msrke(\mfg)-1,\mfu)_{\mmax}  \ar@{->}[r] & \mrmax_{\msrke(\mfg)-1}(\Phi) }.  \]
We find, for type $A_{n}$ the ordering exists for n sufficiently large.

Since the set $\mrmax_{\msrke(\mfg)-1}(\Phi)$  is tractable, it will be determined within the initial step.
The calculation of $\mbE(\msrke(\mfg)-1,\mfu)$ then proceeds via three steps.  First, we determine 
$\mbE(\msrke(\mfg)-1,\mfu)_{\mmax}$ as a set by using the map $\mrlt$. Second,
we prove that the elements of $\mbE(\msrke(\mfg)-1,\mfu)_{\mmax}$  are given by the combinatorics of the
root system of $G$, which largely relies on Malcev's linear algebraic approach. Finally, we have to utilize
the result on $\mbE(\msrke(\mfg),\mfu)$ to understand the elements of $\mbE(\msrke(\mfg)-1,\mfu)$
which are not in $\mbE(\msrke(\mfg)-1,\mfu)_{\mmax}$.
The main result of this paper is:

\bigskip
\begin{thm*}
Let $G$ be a standard simple algebraic $\Bk$-group with root system $\Phi$ of type $A_{n}\;(n\geq 5)$ and 
$\mfg:=\mrlie(G)$. Then the irreducible components of $\mbE(\msrke(\mfg)-1,\mfg)$ can be characterized as follows:
\end{thm*}
\def\arraystretch{2}
\begin{table}[ht]
\centering
\begin{tabular}[b]{|c|c|c|c|}
\hline
Type  & \parbox{60pt}{\vspace{8pt}\centering Restrictions on rank\vspace{8pt}} & Irreducible components  \\
\hline \hline
$A_{2m+1}$  & $m\geq 2$ & $G.\mrlie(\Phir_{m})$, $G.\mrlie(\Phir_{m+2})$, $G.\mbE(\msrke(\mfg)-1,\mrlie(\Phir_{m+1}))$\\
\hline
$A_{2m}$   & $m\geq 3$ & $G.\mbE(\msrke(\mfg)-1,\mrlie(\Phir_{m}))$, $G.\mbE(\msrke(\mfg)-1,\mrlie(\Phir_{m+1}))$ \\
\hline
\end{tabular}
\caption{Characterization}
\end{table}
\def\arraystretch{1}
\bigskip
\noindent

\begin{Remarks}
(1). In \cite{PS} the authors have shown that $\mbE(\msrke(\mfg),\mfg)$ is a finite disjoint union of partial flag
varieties unless $G$ is of type $A_{2}$, which differs from the above result.\\
(2). We list the results of $\mbE(\msrke(\mfg)-1,\mfg)$ for $A_{n}\;(n\leq 4)$ in the following. The reference we give is the
paper \cite{War} of Warner, in which the author discuss the irreducibility of $\mbE(r,\mathfrak{gl}_{n})$ in section 5.

\def\arraystretch{2}
\begin{table}[ht]
\centering
\begin{tabular}[b]{|c|c|c|c|}
\hline
Type  & $\msrke(\mfg)$ & $\mbE(\msrke(\mfg)-1,\mfg)$  \\
\hline \hline
$A_{2}$  & $2$ & irreducible \\
\hline
$A_{3}$   & $4$ & irreducible  \\
\hline
$A_{4}$   & $6$ & unknown  \\
\hline
\end{tabular}
\caption{small rank cases}
\end{table}
\def\arraystretch{1}
\end{Remarks}

{\bf Acknowledgement.} 
The results of this paper are part of the author's doctoral thesis, which he was writing
at the University of Kiel. He would like to thank his advisor, Rolf Farnsteiner, for his continuous support. 
Furthermore, he thanks the members of his working group for proofreading the paper and the referee for his useful  comments.

\section{Preliminaries}
\subsection{Parabolic system}{\label{para-sys}}
We assume that $G$ is a simple algebraic $\Bk$-group with irreducible root system $\Phi$. The interested reader
may consult \cite{Bor}\cite{Bou}\cite{Car}\cite{Spr} for the theory of algebraic groups.
Let $U_{\alpha}$ be the root subgroup corresponding to a root $\alpha$, and $B=\langle U_{\alpha},T \ ; \
\alpha \in \Phi^+\rangle $ be a Borel subgroup of $G$ containing $T$. Initially, we study
the Weyl group $\msW$ in tandem with an irreducible root system $\Phi$. 
Let $\Delta:=\set{\alpha_{1},\ldots,\alpha_{n}}$ be the set of positive simple roots, 
and $I$ be a subset of $\Delta$. We define 
\[   \Phi_{I}:=\Phi\cap\sum_{\alpha\in I}\mathbb{Z}\alpha  \]
to be the parabolic subsystem of roots, and 
\[  \msW_{I}:=\langle s_{\alpha} \ ; \ \alpha\in I \rangle     \]
to be the standard parabolic subgroup of $\msW$\;(see \cite{MT} for details). 
Then subgroups of the form 
$P_{I}:=B\msW_{I}B=\langle T,U_{\alpha}  \ ; \ \alpha\in \Phi^{+}\cup 
\Phi_{I} \rangle$ are called standard parabolic subgroups of $G$. The Levi decomposition 
$P_{I}=L_{I}\ltimes R_{u}(P_{I})$ decomposes $P_{I}$ into a semi-direct product of its Levi factor $L_{I}$ 
and the unipotent radical $R_{u}(P_{I})$, with the latter being generated by root subgroups $\set{U_{\alpha}
\ ; \ \alpha\in \Phi^{+}\setminus\Phi_{I}^{+}}$. Influenced by this, we set $S:=\Delta\setminus I$ and 
then define
\begin{align*}
    \Phi_{S}^{\mrrad}=\Phi^{+}\setminus\Phi_{I}^{+}
\end{align*}                              
to be the set of positive roots that cannot be written as a linear combination 
of the simple roots not in $S$. If $S=\set{\alpha_{i}}$, then we simply write 
$\Phi_{i}^{\mrrad}$ instead of $\Phi_{\set{\alpha_{i}}}^{\mrrad}$.

\subsection{Maximal subsets for type A}
Suppose that $G$ is of type $A_{n}$. The roots of $A_{n}$ are the integer vectors in $\mathbb{R}^{n+1}$ of
length $\sqrt{2}$ for which the coordinates sum to 0. Let $\{\epsilon_{i} \ ; \ 1\leq i\leq n+1\}$ be the standard
basis of $\mathbb{R}^{n+1}$. We denote by 
\[  \Phi=\set{\epsilon_{i}-\epsilon_{j} \ ; \  i\neq j, 1\leq i,j\leq n+1} \]
the corresponding set of roots, and by $\Delta=\set{\alpha_{1},\ldots,\alpha_{n}}$ the base of $\Phi$ where 
$\alpha_{i}=\epsilon_{i}-\epsilon_{i+1}$. There is a bijection $\phi$ from the set of non-trivial proper subsets of
$\set{1,\ldots,n+1}$ to the set of maximal subsets of commuting roots of $\Phi$ by sending $J$ to
$\phi(J):=\set{\epsilon_{i}-\epsilon_{j} \ ; \  i\in J, j\notin J}$, and
the condition $J <\set{1,\ldots,n+1}\setminus J$ on $J$ gives rise to a maximal subset of commuting positive
roots; see \cite[A.1]{PS}.

\begin{Notation}
Type $A_{n}$. 
\begin{align*}
& n=2m+1,& \Phi_{m+1,m+2}^{\mathrm{odd}}:&=\phi(J)\cap \Phi^{+},  \; \mbox{for}\;J=\set{1,\ldots,m,m+2} \\
& n=2m,    & \Phi_{m+1,m+2}^{\mathrm{ev}}:&=\phi(J)\cap \Phi^{+},    \;\mbox{for}\; J=\set{1,\ldots,m,m+2}  \\
&n=2m,     & \Phi_{m,m+1}^{\mathrm{ev}}:&=\phi(J)\cap\Phi^{+} ,         \;\mbox{for}\; J=\set{1,\ldots,m-1,m+1}
\end{align*}
\end{Notation}

\begin{Theorem}
Keep the notations as above and set $\mfg:=\mrlie(G)$. Then the elements of the set $\mrmax_{\msrke(\mfg)-1}(\Phi)$ 
are given as follows:
\end{Theorem}
\def\arraystretch{2}
\begin{table}[ht]{\label{sub-max}}
\centering
\begin{tabular}[b]{|c|c|c|c|}
\hline
Type  & \parbox{60pt}{\vspace{8pt}\centering Restrictions on rank\vspace{8pt}} &   $\mrmax_{\msrke(\mfg)-1}(\Phi)$\\ 
\hline \hline
$A_{2m+1}$  & $m\geq 0$ & $\Phir_{m}$, $\Phir_{m+2}$, $\Phi_{m+1,m+2}^{\mathrm{odd}}$\\
\hline
$A_{2m}$   & $m\geq 1$ & $\Phi_{m+1,m+2}^{\mathrm{ev}}$, $\Phi_{m,m+1}^{\mathrm{ev}}$ \\
\hline
\end{tabular}
\caption{Maximal subset: order $\msrke(\mfg)-1$}
\end{table}
\def\arraystretch{1}
\begin{proof}
Let $M(A)\in \mrmax_{\msrke(\mfg)-1}(\Phi)$. Notice that the maximal dimension $\msrke(\mfg)$ 
is $(m+1)^{2}$ (resp. $m(m+1))$ when $n=2m+1$ (resp. $n=2m$). 
If $M(A)$ is still maximal in $\Phi$, then 
$M(A)=\phi(J)$ for certain $J$. By letting $|M(A)|=|\phi(J)|=|J|(n+1-|J|)$ equal $(m+1)^{2}-1$ 
when $n=2m+1$ and equal $m(m+1)-1$ when $n=2m$, 
we get $|J|=m,m+2$ for $n=2m+1$, and there is no solution for $n=2m$. Continuing the consideration, if 
$|J|=m$ or $m+2$ then $M(A)$ has to be $\Phir_{m}$ or $\Phir_{m+2}$ 
respectively. 

Alternatively, $M(A)$ is maximal in $\Phi^{+}$ but not in $\Phi$.
Then $M(A)\subset \phi(J)$ for some $J$ with 
\[ |\phi(J)|=\msrke(\mfg) \; \mbox{and}\; |\phi(J)\cap \Phi^{+}|=\msrke(\mfg)-1.\]
One gets $M(A)$ equals $\Phi_{m+1,m+2}^{\mathrm{odd}}$
when $n=2m+1$, and equals $\Phi_{m+1,m+2}^{\mathrm{ev}}$ or $\Phi_{m,m+1}^{\mathrm{ev}}$ when $n=2m$. 
\end{proof}

\begin{Remark}
We recall the set $\mrmax(\Phi)$ for type $A_{n}$, which is calculated by Malcev in \cite{Mal}:
\def\arraystretch{2}
\begin{table}[ht]{\label{max}}
\centering
\begin{tabular}[b]{|c|c|c|c|}
\hline
{Type} & \parbox{60pt}{\vspace{8pt}\centering Restrictions on rank\vspace{8pt}} &   $\mrmax(\Phi)$\\ 
\hline \hline
$A_{2m+1}$  & $m\geq 0$ & $\Phir_{m+1}$\\
\hline
$A_{2m}$   & $m\geq 1$ & $\Phir_{m+1}$, $\Phir_{m}$ \\
\hline
\end{tabular}
\caption{Maximal subset: order $\msrke(\mfg)$}
\end{table}
\def\arraystretch{1}
\end{Remark}

\section{Main result}
Now we concentrate on $G$ being a standard connected simple algebraic $\Bk$-group of type $A_{n}$ with
$\mfg:=\mrlie(G)$.  Let $\Phi$ be the root system of $G$ with positive roots $\Phi^+$. Since $p$ is a good prime for $G$, we have 
$[x_\alpha, x_\beta]=0$ if and only if $\alpha+\beta \notin \Phi$ for $\alpha,\beta\in \Phi$ and their associated 
root vectors $x_\alpha, x_\beta$. Recall that $x_\alpha^{[p]}=0$ for $\alpha\in \Phi$,  one does have an elementary
subalgebra $\mrlie(R):=\mrsp_{\Bk} \{ x_{\alpha} \ ; \ \alpha\in R\}$ when $R$ is a subset of commuting roots.
Let $\mfu$ be the Lie algebra of the unipotent radical. We will show the map 
\[ \xymatrix{\mrlie: \mrmax_{\msrke(\mfg)-1}(\Phi)\ar@{->}[r] & \mbE(\msrke(\mfg)-1,\mfu)_{\mmax}}
; \;R \mapsto \mrlie(R)\]
is surjective up to conjugacy by $G$. This will be done by employing the map (cf. \cite[(3.1.2)]{PS})
\[   \xymatrix{ \mrlt:\mbE(\msrke(\mfg)-1,\mfu)_{\mmax}  \ar@{->}[r] & \mrmax_{\msrke(\mfg)-1}(\Phi) } \]
according to the chosen total ordering.

\subsection{Total ordering for map $\mrlt$}
\bigskip

Suppose that $G$ is of type $A_{2m+1}$. We fix the total ordering $\succeq$ by letting it be
the reverse lexicographic ordering given by $\alpha_{m+1}\prec\alpha_{1}\prec \alpha_{2}\prec\cdots\prec\alpha_{2m+1}$. 
We first show that the map $\mrlt$ is well-defined under such setting for $A_{2m+1}$.
\begin{Lem}{\label{submax-A-2m+1}}
Suppose that $G$ is of type $A_{2m+1}\;(m\geq1)$. If $\mfe \in\mbE(\msrke(\mfg)-1,\mfu)_{\mmax}$, then 
$\mrlt(\mfe)\in\mrmax_{\msrke(\mfg)-1}(\Phi)$ with respect to $\succeq$.
\end{Lem}
\begin{proof}
Assume that $\mrlt(\mfe)\notin \mrmax_{\msrke(\mfg)-1}(\Phi)$, then $\mrlt(\mfe)\varsubsetneq\Phir_{m+1}$
by Table \ref{max}.  Since $\Phi^{+}\setminus\Phir_{m+1}\succ \Phir_{m+1}$,  it implies that all terms of basis 
vectors correspond to the roots lying in $\Phir_{m+1}$. As a result, $\mfe$ is contained in the elementary subalgebra
$\mrlie(\Phir_{m+1})$. Notice that  $\dim \mfe < \dim \mrlie(\Phir_{m+1})$, the containment is proper which contradicts
maximality.
\end{proof}

Now we consider the $\Bk$-group $G$ which is of type $A_{2m}$. We choose the total ordering  $\succeq$ to be the 
reverse lexicographic ordering given by $\alpha_{m+1}\prec\alpha_{m}\prec\alpha_{1}\prec\alpha_{2}\prec\cdots\prec\alpha_{2m}$. 
According to this choice, one can easily check that
\begin{align*}
 \Phi^{+}\setminus(\Phi_{m}^{\mrrad}\cup\Phi_{m+1}^{\mrrad}) \succ\Phi_{m}^{\mrrad}\setminus\Phi_{m+1}^{\mrrad}
 \succ \Phi_{m+1}^{\mrrad}\setminus\Phi_{m}^{\mrrad}\succ\Phi_{m}^{\mrrad}\cap\Phi_{m+1}^{\mrrad}.
\end{align*}

\begin{Lem}{\label{submax-A-2m}}
Suppose that $G$ is of type $A_{2m}$ with $m\geq 3$. If $\mfe\in\mbE(\msrke(\mfg)-1,\mfu)_{\mmax}$, then 
$\mrlt(\mfe)\in \mrmax_{\msrke(\mfg)-1}(\Phi)$ with respect to $\succeq$.
\end{Lem}
\begin{proof}
If $\mrlt(\mfe)\notin \mrmax_{\msrke(\mfg)-1}(\Phi)$, then either 
$\mrlt(\mfe)\varsubsetneq \Phi_{m}^{\mrrad}$, or $\mrlt(\mfe)\varsubsetneq \Phi_{m+1}^{\mrrad}$ 
according to Table \ref{max}. \\
$\bullet$\;Case 1. $\mrlt(\mfe)\varsubsetneq \Phi_{m+1}^{\mrrad}$. 
Then $\Phi_{m+1}^{\mrrad}\setminus \mrlt(\mfe)=\set{\epsilon_{u}-\epsilon_{v}}$ for some $(u,v)$. 
Notice that  $\Phi^{+}\setminus\Phi_{m+1}^{\mrrad}\succ \Phi_{m+1}^{\mrrad}$, thus the reduced echelon 
form basis of $\mfe$ is as follows
\begin{align*}
     x_{\epsilon_{i}-\epsilon_{j}}+ a_{ij}x_{\epsilon_{u}-\epsilon_{v}},\; a_{ij}=0\;\text{if}\;i<u\;\text{or}\;i=u,j>v
\end{align*}
for $1\leq i\leq m+1,m+2\leq j\leq 2m+1$ and $(i,j)\neq (u,v)$. Then it is readily seen that $\mfe\varsubsetneq \mfe\oplus \Bk x_{\epsilon_{u}-\epsilon_{v}}$, and the maximality of 
$\mfe$ leads
to a contradiction. \\
$\bullet$ Case 2. 
$\Phi_{m}^{\mrrad}\setminus \mrlt(\mfe)=\set{\epsilon_{u}-\epsilon_{v}}\subset
\Phi_{m}^{\mrrad}\cap\Phi_{m+1}^{\mrrad}$. 
Then the reduced basis of $\mfe$ consists of elements for $1\leq i\leq m, m+2\leq j\leq 2m+1$ and $(i,j)\neq (u,v)$
\begin{align*}
& x_{ij }=x_{\epsilon_{i}-\epsilon_{j}}+ a_{ij}x_{\epsilon_{u}-\epsilon_{v}},\;a_{ij}=0\;\mbox{if}\;
         i<u\;\emph{or}\;i=u,j>v \\
&y_{i}=x_{\epsilon_{i}-\epsilon_{m+1}}+\sum_{s=m+2}^{2m+1}b_{is}x_{\epsilon_{m+1}-\epsilon_{s}}
       + d_{i}x_{\epsilon_{u}-\epsilon_{v}} .    
\end{align*}       
Now we compute
\begin{align*}
   [y_{i},y_{i^{'}}]=\sum_{s=m+2}^{2m+1} b_{i^{'}s}N_{\epsilon_{i}-\epsilon_{m+1},
  \epsilon_{m+1}-\epsilon_{s}}x_{\epsilon_{i}-\epsilon_{s}}+\sum_{s=m+2}^{2m+1} 
  b_{is}N_{\epsilon_{m+1}-\epsilon_{s},\epsilon_{i^{'}}-\epsilon_{m+1}}x_{\epsilon_{i^{'}}-\epsilon_{s}}.
\end{align*}
As $m\geq 3$, we may take $i\neq i^{'}$, this gives $b_{is}=0$ for all $i$ and $s$. As a result, we will have $\mfe\varsubsetneq \mfe\oplus \Bk x_{\epsilon_{u}-\epsilon_{v}}$, a contradiction. \\
$\bullet$\;Case 3.
$\Phi_{m}^{\mrrad}\setminus \mrlt(\mfe)=\set{\epsilon_{u}-\epsilon_{m+1}}\subseteq \Phi_{m}^{\mrrad}\setminus \Phi_{m+1}^{\mrrad}$. Then the reduced echelon form basis of $\mfe$ is
$x_{\epsilon_{i}-\epsilon_{j}}$ for $1\leq i\leq m$ and $m+2\leq j\leq 2m+1$ together with for 
$1\leq i\leq m$ and $i\neq u$
\begin{align*}
   y_{i}=x_{\epsilon_{i}-\epsilon_{m+1}}+q_{i}x_{\epsilon_{u}-\epsilon_{m+1}} +\sum_{s=m+2}^{2m+1}d_{is}x_{\epsilon_{m+1}-\epsilon_{s}},\;\;\;q_{i}=0\;\mbox{if}\;i<u.
\end{align*}
If $i,i^{'}$ are distinct and different from $u$\;(which is possible as $m\geq 3$), 
then the coefficient of $x_{\epsilon_{i^{'}}-\epsilon_{s}}$ in $[y_{i},y_{i^{'}}]$ is
$N_{\epsilon_{m+1}-\epsilon_{s},\epsilon_{i^{'}}-\epsilon_{m+1}}d_{is}$, so $d_{is}=0$ for all $i$  and $s$.
Thus $\mfe\varsubsetneq \mfe\oplus \Bk x_{\epsilon_{u}-\epsilon_{m+1}}$, a contradiction and
we finish the proof.
\end{proof}

\subsection{Surjectivity for map $\mrlie$}
\begin{Thm}{\label{con-A-2m+1}}
Suppose that $G$ is of type $A_{2m+1}$ with $m\geq 2$. If $\mfe\in\mbE(\msrke(\mfg)-1,\mfu)$ satisfies
$\mrlt(\mfe)=\Phir_{m},\Phir_{m+2}$ or $\Phi_{m+1,m+2}^{\mathrm{odd}}$ then
$\mfe=\mrlie(\Phir_{m}), \mrlie(\Phir_{m+2})$ or 
$\mrlie(\Phi_{m+1,m+2}^{\mathrm{odd}})^{\exp(\ad(ax_{\alpha_{m+1}}) )}$ for some $a$ respectively.
\end{Thm}
\begin{proof}
$\bullet$ Case 1. $\mrlt(\mfe)=\Phir_{m}$. We write the reduced echelon form basis for $\mfe$
\begin{align*}
&x_{ij}=x_{\epsilon_{i}-\epsilon_{j}}, 1\leq i\leq m\;\text{and}\; m+2\leq j\leq 2m+2 \\
&y_{i}=x_{\epsilon_{i}-\epsilon_{m+1}}+\sum_{s=1}^{i-1}\sum_{t=s+1}^{m} a_{ist}x_{\epsilon_{s}-\epsilon_{t}}
+\sum_{r=m+2}^{2m+2} b_{ir}x_{\epsilon_{m+1}-\epsilon_{r}}, 1\leq i\leq m
\end{align*}
Let $1\leq i\leq m$ and $2\leq t \leq m$, the coefficient of $x_{\epsilon_{s}-\epsilon_{j}}$ in
$[y_{i},x_{tj}]$ is $a_{ist}N_{\epsilon_{s}-\epsilon_{t},\epsilon_{t}-\epsilon_{j}}$, this gives
$a_{ist}=0$ for all $i,s$ and $t$. If $i,j\leq m$ are distinct, then the coefficient of $x_{\epsilon_{j}-\epsilon_{r}}$ in
$[y_{i},y_{j}]$ is $b_{ir}N_{\epsilon_{m+1}-\epsilon_{r},\epsilon_{j}-\epsilon_{m+1}}$. 
As $m\geq 2$, this gives all $b_{ir}=0$. Therefore, we have $\mfe=\mrlie(\Phir_{m})$. \\
$\bullet$ Case 2. $\mrlt(\mfe)=\Phir_{m+2}$. The reduced echelon form basis of $\mfe$ is of the form
\begin{align*}
&x_{ij}=x_{\epsilon_{i}-\epsilon_{j}}+\sum_{s=1}^{i-1}a_{ijs} x_{\epsilon_{s}-\epsilon_{m+2}}, 
1\leq i\leq m+1\;\text{and}\;m+3\leq j\leq 2m+2 \\
&y_{j}=x_{\epsilon_{m+2}-\epsilon_{j}}+\sum_{s=1}^{m+1}\sum_{t=s+1}^{m+2} b_{jst}x_{\epsilon_{s}-\epsilon_{t}}, m+3\leq j\leq 2m+2
\end{align*}
Let $m+3\leq j,j^{'}\leq 2m+2$ and $2\leq t\leq m+1$. If $j$ and $j^{'}$ are distinct, 
then the coefficient 
of $x_{\epsilon_{s}-\epsilon_{j^{'}}}$ in $[y_{j},x_{tj^{'}}]$ is 
$b_{jst}N_{\epsilon_{s}-\epsilon_{t},\epsilon_{t}-\epsilon_{j^{'}}}$, 
it gives $b_{jst}=0$ for all $j,s$ and $t <m+2$ as $m\geq 2$.
Then the coefficient of $x_{\epsilon_{s}-\epsilon_{j}}$ in $[y_{j},x_{ij^{'}}]$ is 
$a_{ij^{'}s}N_{\epsilon_{m+2}-\epsilon_{j},\epsilon_{s}-\epsilon_{m+2}}$, this implies
$a_{ijs}=0$ for all $i,j,s$. It remains to consider $b_{js(m+2)}$. 
If $m+3\leq i,j\leq 2m+2$ are distinct, then the coefficient of $x_{\epsilon_{s}-\epsilon_{i}}$ 
in $[y_{i},y_{j}]$ is $b_{js(m+2)}N_{\epsilon_{m+2}-\epsilon_{i},\epsilon_{s}-\epsilon_{m+2}}$. 
According to this together with $m\geq 2$, we get $b_{js(m+2)}=0$ for all $j$ and $s$.
Therefore, we have $\mfe=\mrlie(\Phir_{m+2})$.\\
$\bullet$ Case 3. $\mrlt(\mfe)=\Phi_{m+1,m+2}^{\mathrm{odd}}$. The reduced echelon form basis of 
$\mfe$  consists of 
\begin{align*}
&x_{ij}=x_{\epsilon_{i}-\epsilon_{j}} +\sum_{s=1}^{i-1}a_{ijs}x_{\epsilon_{s}-\epsilon_{m+2}} \\
&y_{i}=x_{\epsilon_{i}-\epsilon_{m+1}}+\sum_{s=1}^{i-1}\sum_{t=s+1}^{m}b_{ist}x_{\epsilon_{s}-\epsilon_{t}}
  +\sum_{r=m+2}^{2m+2}d_{ir}x_{\epsilon_{m+1}-\epsilon_{r}}+
  \sum_{r=1}^{m}k_{ir}x_{\epsilon_{r}-\epsilon_{m+2}} \\
&z_{j}=x_{\epsilon_{m+2}-\epsilon_{j}}+\sum_{s=1}^{m-1}\sum_{t=s+1}^{m}h_{jst}x_{\epsilon_{s}-\epsilon_{t}}+
  \sum_{r=m+2}^{2m+2}\ell_{jr}x_{\epsilon_{m+1}-\epsilon_{r}}+
  \sum_{r=1}^{m}\xi_{jr}x_{\epsilon_{r}-\epsilon_{m+2}}
\end{align*}
where $1\leq i\leq m$ and $m+3\leq j\leq 2m+2$.
By the same argument as before we deduce that $b_{ist}
=h_{jst}=0$ for all $s$ and $t$. If $i,j\leq m$ are distinct, then the coefficient of $x_{\epsilon_{j}
-\epsilon_{r}}$ in $[y_{i},y_{j}]$ is $d_{ir}N_{\epsilon_{m+1}-\epsilon_{r},\epsilon_{j}-\epsilon_{m+1}}$.
As $m\geq 2$, this gives $d_{ir}=0$ and the argument can also be  applied to $z_{j}$ which ensures that
$\xi_{jr}=0$. 
Let $\lambda=-k_{11}N_{\epsilon_{m+1}-\epsilon_{m+2}, \epsilon_{1}-\epsilon_{m+1}}$.
Conjugation by $\exp(\ad(\lambda  x_{\alpha_{m+1}}))$ to $\mfe$ ensures that the image 
of $y_{1}$ has no term $x_{\epsilon_{1}-\epsilon_{m+2}}$. We may assume $k_{11}=0$.
We compute the coefficient of $x_{\epsilon_{1}-\epsilon_{r}}$ in $[y_{1},z_{j}]$ 
which is $\ell_{jr}N_{\epsilon_{1}-\epsilon_{m+1},\epsilon_{m+1}-\epsilon_{r}}$, 
giving $\ell_{jr}=0$ for all $j$ and $r$. Then the coefficient
of $x_{\epsilon_{r}-\epsilon_{j}}$ in $[y_{i},z_{j}]$ is $k_{ir}N_{\epsilon_{r}-\epsilon_{m+2},
\epsilon_{m+2}-\epsilon_{j}}$, this gives $k_{ir}=0$ for all $i$ and $r$, and this also applies to 
$[x_{ij},z_{j^{'}}]$ from which we can get $a_{ijs}=0$. 
As a result, we get $\mfe=\mrlie(\Phi_{m+1,m+2}^{\mathrm{odd}})^{\exp(\ad(ax_{\alpha_{m+1}}))}$ where
$a=-\lambda$.
\end{proof}

\begin{Thm}{\label{con-A-2m}}
Suppose that $G$ is of type $A_{2m}$ with $m\geq 3$. 
If $\mfe\in \mbE(\msrke(\mfg)-1,\mfu)$ satisfies $\mrlt(\mfe)=\Phi_{m,m+1}^{\mathrm{ev}}$ or 
$\Phi_{m+1,m+2}^{\mathrm{ev}}$ then there exists some $a$ such that
$\mfe=\mrlie(\Phi_{m,m+1}^{\mathrm{ev}})^{\exp(\ad(ax_{\alpha_{m}}))}$ or
$\mrlie(\Phi_{m+1,m+2}^{\mathrm{ev}})^{\exp(\ad(ax_{\alpha_{m+1}}))}$ respectively.
\end{Thm}
\begin{proof}
$\bullet$\;Case 1.\;$\mrlt(\mfe)=\Phi_{m,m+1}^{\mathrm{ev}}$.
Then the reduced echelon form basis of $\mfe$ is $x_{\epsilon_{i}-\epsilon_{j}}$ for $1\leq i\leq m-1$ and $m+2\leq j\leq 2m+1$ and
\begin{align*}
      &y_{j}=x_{\epsilon_{m+1}-\epsilon_{j}}+\sum_{s=m+2}^{2m+1}a_{js}x_{\epsilon_{m}-\epsilon_{s}}\\
      &z_{i}=x_{\epsilon_{i}-\epsilon_{m}}+\sum_{u=1}^{i-1}\sum_{v=u+1}^{m-1}b_{iuv}x_{\epsilon_{u}-\epsilon_{ v}}+\sum_{s=1}^{m}c_{is}
      x_{\epsilon_{s}-\epsilon_{m+1}}+\sum_{s=m+2}^{2m+1}d_{is}x_{\epsilon_{m}-\epsilon_{s}}
\end{align*}
where $m+2\leq j\leq 2m+1$ for $y_{j}$ and $1\leq i<m$ for $z_{i}$. 
Let $\lambda=-a_{(m+2)(m+2)}N_{\epsilon_{m}-\epsilon_{m+1},\epsilon_{m+1}-\epsilon_{m+2}}$.
Using conjugation given by
$\exp(\ad(\lambda x_{\epsilon_{m}-\epsilon_{m+1}}))$ to $\mfe$,  we have explicitly
\begin{align*}
  \exp(\ad(\lambda x_{\epsilon_{m}-\epsilon_{m+1}}))(y_{m+2})=x_{\epsilon_{m+1}-\epsilon_{m+2}}
   +\sum_{s= m+3}^{2m+1}a_{(m+2)s}x_{\epsilon_{m}-\epsilon_{s}},
\end{align*}
which allows us to assume $a_{(m+2)(m+2)}=0$.
Then we compute for $1\leq i< m$
\begin{align*}
 [y_{m+2},z_{i}]=\sum_{s=1}^{m}c_{is}N_{\epsilon_{m+1}-\epsilon_{m+2},\epsilon_{s}
     -\epsilon_{m+1}}x_{\epsilon_{s}-\epsilon_{m+2}}+
     \sum_{s=m+3}^{2m+1}a_{(m+2)s}N_{\epsilon_{m}-\epsilon_{s},\epsilon_{i}-\epsilon_{m}}x_{\epsilon_{i}
     -\epsilon_{s}}.
\end{align*}
Notice that these items $x_{\epsilon_{s}-\epsilon_{m+2}}$ and $x_{\epsilon_{i}-\epsilon_{s}}$ are different,
so $c_{is}=0$ for all $i$ and $s$ and $a_{(m+2)s}=0$ for all $s$. Further we can get $a_{js}=0$ for all $j$ and $s$ by
seeing $[y_{j},z_{1}]=0$.
When $1 < v<m$, we compute the coefficient of $x_{\epsilon_{u}-\epsilon_{m+2}}$ in 
$[x_{\epsilon_{v}-\epsilon_{m+2}},z_{i}]$, that is
$N_{\epsilon_{v}-\epsilon_{m+2},\epsilon_{u}-\epsilon_{v}}b_{iuv}$. 
This gives $b_{iuv}=0$ for all $i,u$ and $v$.
If $1\leq i,i^{'}<m$ are distinct, then the coefficient of $x_{\epsilon_{i}-\epsilon_{s}}$ in 
$[z_{i},z_{i^{'}}]$ is
$N_{\epsilon_{i}-\epsilon_{m},\epsilon_{m}-\epsilon_{s}}d_{i^{'}s}$. As $m\geq 3$, this  gives $d_{is}=0$ 
for all $i,s$, and consequently $\mfe=\mrlie(\Phi_{m,m+1}^{\mathrm{ev}})^{\exp(\ad(ax_{\alpha_{m}}))}$ for 
$a=-\lambda$.\\
$\bullet$\;Case 2.\;$\mrlt(\mfe)=\Phi_{m+1,m+2}^{\mathrm{ev}}$. 
We write the reduced basis for $1\leq i\leq m$ and $m+3\leq j \leq 2m+1$
\begin{align*}
&x_{ij}=x_{\epsilon_{i}-\epsilon_{j}}+\sum_{t=1}^{i-1}a_{ijt}x_{\epsilon_{t}-\epsilon_{m+2}},\\
&y_{i}=x_{\epsilon_{i}-\epsilon_{m+1}}+\sum_{s=m+2}^{2m+1}b_{is}x_{\epsilon_{m+1}-\epsilon_{s}}
           +\sum_{s=1}^{m}c_{is}x_{\epsilon_{s}-\epsilon_{m+2}},\\
&z_{j}=x_{\epsilon_{m+2}-\epsilon_{j}}+\sum_{u=1}^{m-1}\sum_{v=u+1}^{m}d_{juv}x_{\epsilon_{u}-\epsilon_{v}}
 +\sum_{s=m+2}^{2m+1}f_{js}x_{\epsilon_{m+1}-\epsilon_{s}}+\sum_{s=1}^{m}k_{js}x_{\epsilon_{s}-\epsilon_{m+2}}.
\end{align*}
If $m+3\leq j, j^{'}\leq 2m+1$ and $j\neq j^{'}$, choose $1< v<m+1$, then we compute
\begin{align*}
  [x_{vj},z_{j^{'}}]=\sum_{u=1}^{v-1}d_{j^{'}uv}N_{\epsilon_{v}-\epsilon_{j},\epsilon_{u}-\epsilon_{v}}
                x_{\epsilon_{u}-\epsilon_{j}} +\sum_{t=1}^{i-1}
                a_{vjt}N_{\epsilon_{t}-\epsilon_{m+2},\epsilon_{m+2}-\epsilon_{j^{'}}}x_{\epsilon_{t}-\epsilon_{j^{'}}}.
\end{align*}
As $m\geq 3$, this gives $d_{juv}=0$ and consequently $a_{ijt}=0$ by seeing the coefficient of 
$x_{\epsilon_{t}-\epsilon_{j}}$ in $[x_{ij},z_{j}]$. If $1\leq i ,i^{'}\leq m$ and $i\neq i^{'}$, then the coefficient of
$x_{\epsilon_{i}-\epsilon_{s}}$ in $[y_{i},y_{i^{'}}]$ is $N_{\epsilon_{i}-\epsilon_{m+1},
   \epsilon_{m+1}-\epsilon_{s}}b_{i^{'}s}$, so $b_{is}=0$ for all $i$ and $s$. Now let $\xi=
   -c_{11}N_{\epsilon_{m+1}-\epsilon_{m+2},\epsilon_{1}-\epsilon_{m+1}}$, 
conjugation given by $\exp(\ad(\xi x_{\epsilon_{m+1}-\epsilon_{m+2}}))$ lets us assume that $c_{11}=0$.
Then we compute
\begin{align*}
    [y_{1},z_{j}]=\sum_{s=m+2}^{2m+1}f_{js}N_{\epsilon_{1}-\epsilon_{m+1},\epsilon_{m+1}-\epsilon_{s}}
                         x_{\epsilon_{1}-\epsilon_{s}}+
                        \sum_{s=2}^{m}c_{1s}N_{\epsilon_{s}-\epsilon_{m+2},
                         \epsilon_{m+2}-\epsilon_{j}}x_{\epsilon_{s}-\epsilon_{j}}
\end{align*}
It follows that $c_{1s}$ and $f_{js}$ are zero. Further $c_{is}=0$ for all $i$ and $s$ by computing $[y_{i},z_{1}]$.
Finally, the coefficient of $x_{\epsilon_{s}-\epsilon_{i}}$ in $[z_{i},z_{j}]$ for $i\neq j$ is 
$N_{\epsilon_{m+2}-\epsilon_{i},\epsilon_{s}-\epsilon_{m+2}}k_{js}$, so $k_{js}=0$ for all $j$ and $s$. 
Now we have $\mfe=\mrlie(\Phi_{m+1,m+2}^{\mathrm{ev}})^{\exp(\ad(ax_{\alpha_{m+1}}))}$ for $a=-\xi$ and complete the proof.
\end{proof}

\subsection{Irreducible components}
\begin{Definition}(\cite[Definition 2.10]{PS})
We say $R\subset \Phi^{+}$ is an \emph{ideal} if $\alpha+\beta \in R$ whenever $\alpha\in R, \beta\in \Phi^{+}$
and $\alpha+\beta\in \Phi^{+}$. 
\end{Definition}
\begin{Lem}{\label{submax}}
Suppose that $G$ is of type $A_{n}$ with $n\geq 5$. Then 
\begin{align*}
 \mbE(\msrke(\mfg)-1,\mfu)_{\mmax} \subseteq \bigcup_{R\;\mathrm{an\;ideal}}G.\mrlie(R),
\end{align*}
and the ideals occurring here for each type are listed in the third column of the following Table 
\end{Lem}
\begin{table}[ht]
\centering
\begin{tabular}[b]{|c|c|c|c|}
\hline
Type  & \parbox{60pt}{\vspace{8pt}\centering Restrictions on rank\vspace{8pt}}  & Ideal $R$  \\
\hline \hline
$A_{2m+1}$  & $m\geq 2$ &$\Phi_{m}^{\mrrad},\Phi_{m+2}^{\mrrad},\Phi_{m+1}^{\mrrad}\setminus\set{\alpha_{m+1}}$  \\
\hline
$A_{2m}$   & $m \geq 3$ &$\Phi_{m}^{\mrrad}\setminus\set{\alpha_{m}}, \Phi_{m+1}^{\mrrad}\setminus\set{\alpha_{m+1}}$ \\
\hline
\end{tabular}
\caption{Ideals for Lemma \ref{submax}  }
\label{ideals1}
\end{table}
\begin{proof}
For type $A_{2m+1}$, $\Phir_{m}$ and $\Phir_{m+2}$ both are ideals, and $\Phi_{m+1,m+2}^{\odd}$ can be
conjugated to $\Phi_{m+1}^{\mrrad}\setminus\set{\alpha_{m+1}}$ by a simple reflection $s_{m+1}$.
For type $A_{2m}$, $\Phi_{m,m+1}^{\ev}$ is conjugate to $\Phir_{m}\setminus\set{\alpha_{m}}$ by $s_{m}$, 
and $\Phi_{m+1,m+2}^{\ev}$ is conjugate to $\Phir_{m+1}\setminus\set{\alpha_{m+1}}$ by $s_{m+1}$.  Then 
it is a summarization of Theorem \ref{con-A-2m+1} and  Theorem {\ref{con-A-2m}}.
\end{proof}

\begin{Cor}{\label{union}}
Let $G$ be a standard simple algebraic $\Bk$-group with root system $A_{n}\;(n\geq 5)$. Then
\[ (\ast)\ \ \ \ \mbE(\msrke(\mfg)-1,\mfg)= \bigcup_{R\; \mathrm{an\; ideal}}G.\mrlie(R) \cup  \bigcup_{I \;\mathrm{an\; ideal}}G.\mbE(\msrke(\mfg)-1,\mrlie(I)) \]
\end{Cor}
is the union of irreducible closed subsets, where $R$ is taken from Table \ref{ideals1} and $I$ is given as follows:
\begin{table}[ht]
\centering
\begin{tabular}[b]{|c|c|c|c|}
\hline
Type  & \parbox{60pt}{\vspace{8pt}\centering Restrictions on rank\vspace{8pt}}  & Ideal $I$  \\
\hline \hline
$A_{2m+1}$  & $m\geq 2$ &$\Phi_{m+1}^{\mrrad}$  \\
\hline
$A_{2m}$   & $m\geq 3$ &$\Phi_{m}^{\mrrad}, \Phi_{m+1}^{\mrrad}$\\
\hline
\end{tabular}
\caption{Ideals for Corollary \ref{union}}
\label{ideals2}
\end{table}
\begin{proof}
Let $R$ be an ideal in Table 4 and $I$ be an ideal in Table 5. 
We define $X_{1}:=\mrlie(R), X_{2}:=\mbE(\msrke(\mfg)-1,\mrlie(I))$ and 
$Y:= \mbE(\msrke(\mfg)-1,\mfg)$.
Since $X_{2}$ is a projective variety, it is complete, implying that $X_{2}$ is closed in $Y$. 
Since $R$ and $I$ are ideals, it follows that $X_{i}$ is stabilized by a  parabolic subgroup of $G$ 
for $i\in \{1,2\}$ respectively. 
By \cite[Theorem 4.9]{PS} for $X_{1}$ and \cite[Proposition 0.15]{Hum2}  for $X_{2}$, we have 
\[ G.X_{i} \; \text{is closed in}\; Y, \mbox{where}\; i \in \{1,2\}. \]
Since $Y$ is a $G$-variety, $G.X_{1}$ is irreducible as a $G$-orbit. Since 
$\mrlie(I)$ is an elementary subalgebra of $\mfg$, it follows that
$X_{2} =\Gr_{\msrke(\mfg)-1}(\mrlie(I))(\Bk)$ is the Grassmannian which is irreducible. 
Then  $G.X_{2}$ as the image of $X_{2}$ under $G$ is irreducible.
As a result, the right hand of $(\ast)$ is the union of irreducible closed subsets.

By utilizing Lemma \ref{submax} along with \cite[Sect. 3.2/3.4]{PS}, we have
\[  \mbE(\msrke(\mfg)-1,\mfu)\subset \bigcup_{R\; \mathrm{an\; ideal}}G.\mrlie(R) \cup  \bigcup_{I \;\mathrm{an\; ideal}}G.\mbE(\msrke(\mfg)-1,\mrlie(I)). \]
Therefore, we arrive at the equality of $(\ast)$ according to  \cite[Lemma 2.2]{Pre2}.
\end{proof}

\begin{Lem}{\label{W-con}}
Let $\mfe$ be an element of $\mbE(r,\mfg)$ and $R$ be an ideal of commuting roots with $|R|=r$.
Assume that there is $g\in G$, satisfying
\[    g.\mfe =\mrlie(R).  \]
Then $\mrlt(\mfe)$ and $R$ are conjugate by an element of $\msW$.
\end{Lem}
\begin{proof}
By Bruhat decomposition of $G$, there exist elements $b, b^{'}\in B$ and $ w\in \msW$ such that $g=b\dot{w}b^{'}$
where $\dot{w}$ is an element of $N_{G}(T)$ whose image in the Weyl group $\msW$ is $w$.
Since $g.\mfe=\mrlie(R)$, we have $\dot{w}b^{'}.\mfe=b^{-1}.\mrlie(R)$. Notice that $R$ is an ideal, 
implying $B\subset \mrstab_{G}(\mrlie(R))$.
Thus $\dot{w}b^{'}.\mfe=\mrlie(R)$ and 
\[ (\ast\ast) \ \ \ b^{'}.\mfe=\dot{w}^{-1}.\mrlie(R)=\mrlie(w^{-1}.R). \]
Observe that the action of $U_{\alpha}$ on $\mfe$ is lower triangular 
with respect to $\succeq$ for $\alpha \in \Phi^{+}$.
Then the equality $(\ast\ast)$ gives $\mrlt(\mfe)=w^{-1}.R$, as desired.
\end{proof}

\begin{Thm}{\label{irreduciblecom}}
Let $G$ be a standard simple algebraic $\Bk$-group with root system $\Phi$ of type $A_{n}\;(n\geq 5)$. 
Then the irreducible components of $\mbE(\msrke(\mfg)-1,\mfg)$ can be characterised; see Table \ref{irr-com-1}.
\end{Thm}
\def\arraystretch{2}
\begin{table}[ht]
\centering
\begin{tabular}[b]{|c|c|c|c|}
\hline
Type  & \parbox{60pt}{\vspace{8pt}\centering Restrictions on rank\vspace{8pt}}  & Irreducible components  \\
\hline \hline
$A_{2m+1}$  & $m\geq 2$ & $G.\mrlie(\Phir_{m})$, $G.\mrlie(\Phir_{m+2})$, $G.\mbE(\msrke(\mfg)-1,\mrlie(\Phir_{m+1}))$\\
\hline
$A_{2m}$   & $m\geq 3$ & $G.\mbE(\msrke(\mfg)-1,\mrlie(\Phir_{m}))$, $G.\mbE(\msrke(\mfg)-1,\mrlie(\Phir_{m+1}))$ \\
\hline
\end{tabular}
\caption{Irreducible components for Theorem \ref{irreduciblecom} }
\label{irr-com-1}
\end{table}
\def\arraystretch{1}
\begin{proof}
By Corollary \ref{union}, it suffices to check the maximality of each irreducible closed subset.
Let $R_{v}$ be an ideal of commuting roots of order $\msrke(\mfg)-1$ for $v \in J:=\set{1,2}$.
Let $I_{v}\in \mrmax(\Phi)$ be an ideal for $v \in J$.
We will apply Lemma \ref{W-con} to the following three cases for $\set{u,v}=J$:
\begin{enumerate}
\item[(1)]$G.\mrlie(R_{v})\subseteq G.\mrlie(R_{u})$.
\item[(2)]$G.\mrlie(R_{v})\subseteq G.\mbE(\msrke(\mfg)-1,\mrlie(I_{u}))$.
\item[(3)]$G.\mbE(\msrke(\mfg)-1,\mrlie(I_{v})) \subseteq G.\mbE(\msrke(\mfg)-1,\mrlie(I_{u}))$.
\end{enumerate}
We conclude that $R_{u}$ and $R_{v}$ are $\msW$-conjugate from $(1)$.
In $(2)$, we have $\mrlie(R_{v})=g.\mfe$ for some $g\in G$ and $\mfe\in \mbE(\msrke(\mfg)-1,\mrlie(I_{u}))$.
Therefore $R_{v}$ and $\mrlt(\mfe)$ are conjugate by an element of $\msW$.
In $(3)$, let $\gamma$ be the unique positive simple root in $I_{v}$ and $\mfe$ be an element of $\mbE(\msrke(\mfg)-1,\mrlie(I_{u}))$
such that $\mrlie(I_{v} \setminus\set{\gamma})=g.\mfe$ for some $g\in G$.
Then we have  $I_{v}\setminus\set{\gamma}$ and $\mrlt(\mfe)$ are $\msW$-conjugate. 

Now we are in the position to classify the irreducible components for $A_{n}(n\geq 5)$: \\
$\bullet$ Type $A_{2m+1}$. 
(a)\;$G.\mrlie(\Phir_{m+1}\setminus\set{\alpha_{m+1}})$ is not maximal because $\mrlie(\Phir_{m+1}\setminus
\set{\alpha_{m+1}})$ is an element of $\mbE(\msrke(\mfg)-1,\mrlie(\Phir_{m+1}))$.
(b)\;If $G.\mrlie(\Phir_{m})\subseteq G.\mrlie(\Phir_{m+2})$ or $G.
\mrlie(\Phir_{m})\subseteq G.\mbE(\msrke(\mfg)-1,\mrlie(\Phir_{m+1}))$ then
$\Phir_{m}$ is $\msW$-conjugate to $\Phir_{m+2}$ or conjugate to $\mrlt(\mfe(\Phir_{m}))$.
Both cases are impossible when we look at \cite[Lemma 2.6]{PS}, this gives $G.\mrlie(\Phir_{m})$ is maximal. 
(c)\;$G.\mrlie(\Phir_{m+2})$ and $G.\mbE(\msrke(\mfg)-1,\mrlie(\Phir_{m+1}))$
are maximal by the same argument of (b). \\
$\bullet$ Type $A_{2m}$. 
(a)\;$G.\mrlie(\Phir_{m}\setminus\set{\alpha_{m}})$ and $G.\mrlie(\Phir_{m+1}
\setminus\set{\alpha_{m+1}})$ are not maximal since they are contained in $G.\mbE(\msrke(\mfg)-1,\mrlie(R^{'}))$
for $R^{'}=\Phir_{m},\Phir_{m+1}$ respectively.
(b)\;We claim $G.\mbE(\msrke(\mfg)-1,\mrlie(\Phir_{m}))$ and $G.\mbE(\msrke(\mfg)-1,\mrlie(\Phir_{m+1}))$ are maximal.
Without loss of generality, we may assume that 
\[  G.\mbE(\msrke(\mfg)-1,\mrlie(\Phir_{m}))\subseteq G.\mbE(\msrke(\mfg)-1,\mrlie(\Phir_{m+1})).\]
Then $\Phir_{m}\setminus\set{\alpha_{m}}$ is $\msW$-conjugate to $\Phir_{m+1}\setminus\set{\gamma}$ where
$\gamma=\alpha_{1}+\cdots+\alpha_{2m}$ is the highest root, and 
consequently $\Phir_{m}\setminus\set{\gamma}$ and $\Phir_{m+1}\setminus\set{\gamma}$ are conjugate. 
Notice that the Weyl group of $A_{2m}$ is the permutation group $\mfS_{2m+1}$. Let 
$w.\Phir_{m}\setminus\set{\gamma}=\Phir_{m+1}\setminus\set{\gamma}$ for some $w\in \msW$ and 
$m+1\leq j_{0}< 2m+1$. We denote by $w(j_{0})$ the corresponding action for $j_{0}$ when $w$ acts on 
$\set{\epsilon_{i}-\epsilon_{j_{0}}}_{1\leq i<m+1}\subseteq\Phir_{m}\setminus\set{\gamma}$. 
Then $w.\set{\epsilon_{i}-\epsilon_{j_{0}}}_{1\leq i<m+1}\subseteq\set{\epsilon_{i}-
\epsilon_{w(j_{0})}}_{1\leq i<m+2}\subseteq \Phir_{m+1}\setminus\set{\gamma}$ and 
$w.\set{\epsilon_{i}-\epsilon_{r}}_{1\leq i<m+1}\nsubseteq \set{\epsilon_{i}-\epsilon_{w(j_{0})}}_{1\leq i<m+2}$
for $m+1\leq r <2m+1 $ with $r\neq j_{0}$. 
As $m\geq 3$, there exists $j_{0}$ such that $w(j_{0})\neq 2m+1$. Then the equality
$|\set{\epsilon_{i}-\epsilon_{j_{0}}}_{1\leq i<m+1}|<
|\set{\epsilon_{i}-\epsilon_{w(j_{0})}}_{1\leq i<m+2}|$ shows the impossibility.\\
\end{proof}

\begin{Remarks}
We would like to refer the reader to the CAU-thesis \cite{Pan} for other classical types:
\def\arraystretch{2}
\begin{table}[ht]
\centering
\begin{tabular}[b]{|c|c|c|c|}
\hline
Type  & \parbox{60pt}{\vspace{8pt}\centering Restrictions on rank\vspace{8pt}} & Irreducible components  \\
\hline
$B_n$  & $n \geq 5$ & $G.\mbE(\msrke(\mfg)-1,\mrlie(S_{1}))$ \\
\hline
$C_{n}$ & $n\geq 3$ & $G.\mbE(\msrke(\mfg)-1,\mrlie(\Phir_{n}))$\\
\hline
$D_{n}$ & $n\geq 6$ & $G.\mbE(\msrke(\mfg)-1,\mrlie(\Phir_{n-1}))$, $G.\mbE(\msrke(\mfg)-1,\mrlie(\Phir_{n}))$ \\
\hline
\end{tabular}
\caption{Irreducible components for other classical types}
\label{irr-com-2}
\end{table}
\def\arraystretch{1}
\end{Remarks}


\end{document}